\documentclass[a4paper,11pt]{amsart}
\usepackage[colorlinks, linkcolor=blue,anchorcolor=Periwinkle,
    citecolor=blue,urlcolor=Emerald]{hyperref}
\usepackage[all]{xy}
\SelectTips{cm}{}

\usepackage{graphicx}
\usepackage{psfrag}

\usepackage{mathtools}
\usepackage{tikz}
\usepackage{tikz-cd}
\usepackage{adjustbox}
\usepackage{extarrows}

\usepackage{amssymb}
\usetikzlibrary{decorations.pathreplacing}
\usetikzlibrary{matrix,arrows}

\usepackage{enumitem}

\usepackage{stmaryrd}

\newcommand{\ie}{{i.e.}\ }
\newcommand{\cf}{{cf.}\ }

\newcommand{\ko}{\: , \;}
\newcommand{\ul}[1]{\underline{#1}}
\newcommand{\ol}[1]{\overline{#1}}

\numberwithin{equation}{section}

\newtheorem{classification-theorem}[subsection]{Classification Theorem}
\newtheorem{decomposition-theorem}[subsection]{Decomposition Theorem}
\newtheorem{proposition-definition}[subsection]{Proposition-Definition}
\newtheorem{periodicity-conjecture}[subsection]{Periodicity Conjecture}

\newtheorem{theorem}{Theorem}
\numberwithin{theorem}{section}
\newtheorem{thmx}{Theorem}

\newtheorem{lemma}[theorem]{Lemma}
\newtheorem{proposition}[theorem]{Proposition}

\theoremstyle{definition}
\newtheorem{definition}[theorem]{Definition}
\theoremstyle{plain}

\newtheorem{example}[theorem]{Example}
\newtheorem{remarks}[theorem]{Remarks}

\newcommand{\reminder}[1]{}

\renewcommand{\mod}{\mathrm{mod}\,}
\newcommand{\bimod}{\mathrm{bimod}\,}
\newcommand{\sg}{\mathrm{sg}\,}
\newcommand{\CM}{\mathrm{CM}\,}
\newcommand{\cosg}{\mathrm{cosg}\,}

\newcommand{\Mod}{\mathrm{Mod}\,}
\newcommand{\Bimod}{\mathrm{Bimod}\,}
\newcommand{\proj}{\mathrm{proj}\,}

\newcommand{\per}{\mathrm{per}\,}

\newcommand{\add}{\mathrm{add}\,}

\newcommand{\op}{^{op}}

\newcommand{\Tr}{\mathrm{Tr}}

\newcommand{\sym}{\mathrm{sym}}
\newcommand{\Sym}{\mathrm{Sym}\,}

\newcommand{\Z}{\mathbb{Z}}

\newcommand{\C}{\mathbb{C}}

\newcommand{\iso}{\xrightarrow{_\sim}}
\newcommand{\liso}{\xleftarrow{_\sim}}
\newcommand{\id}{\mathbf{1}}
\newcommand{\F}{\mathbb{F}}

\newcommand{\Hom}{\mathrm{Hom}}

\newcommand{\End}{\mathrm{End}}

\newcommand{\GL}{\mathrm{GL}}
\newcommand{\SL}{\mathrm{SL}}

\newcommand{\ten}{\otimes}

\newcommand{\ca}{{\mathcal A}}

\newcommand{\cc}{{\mathcal C}}
\newcommand{\cd}{{\mathcal D}}

\newcommand{\ch}{{\mathcal H}}

\newcommand{\Si}{\Sigma}
\newcommand{\si}{\sigma}

\newcommand{\eps}{\varepsilon}
\newcommand{\del}{\partial}

\renewcommand{\hat}[1]{\widehat{#1}}

\renewcommand{\tilde}[1]{\widetilde{#1}}

\begin{document}

\date{\today}

\title[Singularity categories via higher McKay quivers with potential]{Singularity categories via\\[0.15cm] higher McKay quivers with potential}

\author{Junyang Liu}
\address{School of Mathematical Sciences, University of Science and Technology of China, Hefei 230026, China}
\address{Graduate School of Mathematical Sciences, The University of Tokyo, 3-8-1 Komaba \linebreak Meguro-ku Tokyo 153-8914, Japan}
\email{liuj@imj-prg.fr}
\urladdr{https://webusers.imj-prg.fr/~junyang.liu}

\begin{abstract}
In 2018, Kalck and Yang showed that the singularity categories associated with $3$-dimensional Gorenstein quotient singularities are triangle equivalent (up to direct summands) to small cluster categories associated with McKay quivers with potential. We introduce higher McKay quivers with potential and generalize Kalck and Yang's theorem to arbitrary dimensions. The singularity categories we consider occur as the stable categories of categories of Cohen--Macaulay modules. We refine our description of the singularity categories by showing that these categories of Cohen--Macaulay modules are equivalent to Higgs categories in the sense of Wu. Moreover, we describe the singularity categories in the non-Gorenstein case.
\end{abstract}




\maketitle

\vspace*{-1cm}
\tableofcontents

\section{Introduction}

Singularity categories were introduced and studied by Buchweitz \cite{Buchweitz21} and later also by Orlov \cite{Orlov04, Orlov06, Orlov11}, who related them to Kontsevich's homological
mirror symmetry conjecture. They may be seen as a categorical measure for the complexity of the singularities of a noetherian scheme. The explicit description of the singularity categories associated with commutative Gorenstein rings is a problem that has received much attention over the years. Ginzburg \cite{Ginzburg06} described $3$-dimensional Gorenstein quotient singularities via McKay quivers with potential. By Auslander's algebraic McKay correspondence, the singularity categories associated with Kleinian singularities are triangle equivalent to the $1$-cluster categories associated with Dynkin quivers. In 2015, Amiot--Iyama--Reiten \cite{AmiotIyamaReiten15} constructed triangle equivalences between the singularity categories associated with cyclic quotient singularities and the (higher) cluster categories associated with certain finite-dimensional algebras. In 2016, Thanhoffer de V\"olcsey--Van~den Bergh \cite{ThanhofferVandenBergh16} used McKay quivers with potential to obtain an explicit description of such finite-dimensional algebras for singularities of dimension $3$. Later, Kalck--Yang \cite{KalckYang18} extended this result to non-isolated singularities via small cluster categories.

The dimension `$3$' appears in the quotient singularities because the classical McKay quivers with potential and Ginzburg dg(=differential graded) algebras are of Calabi--Yau dimension $3$. Ginzburg \cite{Ginzburg06} generalized the definition of quivers with potential to arbitrary dimensions. He introduced an important class of dg algebras, which are called (higher-dimensional) Ginzburg dg algebras. Since it is natural to consider higher-dimensional quotient singularities and Ginzburg dg algebras have higher-dimensional analogs, it seems natural to generalize McKay quivers with potential to the higher-dimensional case to describe the singularity categories associated with higher-dimensional quotient singularities.

We introduce higher McKay quivers with potential associated with finite subgroups of $\SL_n$ for arbitrary positive integer $n$. Without assuming the singularities to be isolated, we generalize Kalck--Yang's theorem to arbitrary dimensions. In the case of isolated singularities, we refine their theorem by showing that the corresponding category of (maximal) Cohen--Macaulay modules is equivalent to a Higgs category in the sense of Wu \cite{Wu23a}. Consequently, the singularity category is $(n-1)$-Calabi--Yau and contains a canonical \linebreak $(n-1)$-cluster-tilting object. We also prove the structure theorem in the non-Gorenstein case, \ie $\GL_n$ case. Our proof is inspired by \cite{Ginzburg06}.

The article is organized as follows: In Section~\ref{ss:quivers with potential}, we discuss the Ginzburg dg algebra associated with a quiver with potential. In Section~\ref{ss:singularity categories}, we recall singularity categories and (small) cluster categories.

In Proposition~\ref{prop:minimal model}, we prove that polynomial algebras are quasi-isomorphic to the Ginzburg dg algebras associated with explicit quivers with potential. In Lemma~\ref{lemma:functor property}, we construct a monoidal functor which plays a crucial role in the proof of our main results. In Lemma~\ref{lemma:natural isomorphism}, we give another description of this functor via skew group algebras. Then we state the structure theorem in the $\GL_n$ case as follows.

\begin{thmx}[see Theorem~\ref{thm:GLn} for details] \label{thm:A}
Let $G\subset \GL_n$ be a finite subgroup. Then the singularity category associated with $k[x_1, \ldots, x_n]^G$ is triangle equivalent up to direct summands to the small cosingularity category (\cf Section~\ref{ss:singularity categories}) associated with an explicit dg tensor algebra $A'$. If $k[x_1, \ldots, x_n]^G$ has isolated singularities, then this is a triangle equivalence.
\end{thmx}

In Lemma~\ref{lemma:invariant subalgebra}, we describe the invariant subalgebra and some modules over it via the skew group algebra. In Proposition~\ref{prop:GLn minimal model}, we give the minimal model (in the sense of \cite{ThanhofferVandenBergh16}) of the skew group algebra up to a Morita equivalence. This is the key to proving Theorem~\ref{thm:A}. In Definition~\ref{def:McKay quiver}, we introduce higher McKay quivers with potential for arbitrary finite subgroups of $\SL_n$. The associated $n$-dimensional Ginzburg dg algebra does not depend on various choices of higher McKay quivers with potential, \cf Remarks~\ref{rk:well-defined}. In the $\SL_n$ case, we find a suitable homogeneous $k$-basis of the tensor factor of the dg tensor algebra $A'$ in Theorem~\ref{thm:A} to make $A'$ into an $n$-dimensional Ginzburg dg algebra associated with a higher McKay quiver with potential. Moreover, in the case of isolated singularities, we establish a connection between the category of Cohen--Macaulay modules and a Higgs category in the sense of \cite{Wu23a, KellerWu23}.

\begin{thmx}[see Theorem~\ref{thm:SLn} for details] \label{thm:B}
Let $G\subseteq \SL_n$ be a finite subgroup. Then the singularity category associated with $k[x_1, \ldots, x_n]^G$ is triangle equivalent up to direct summands to the small cluster category (\cf Section~\ref{ss:singularity categories}) associated with an $n$-dimensional Ginzburg dg algebra $\Pi_n(Q', W')$. If $k[x_1, \ldots, x_n]^G$ has isolated singularities, then the category of Cohen--Macaulay $k[x_1, \ldots, x_n]^G$-modules is equivalent to the Higgs category associated with an $n$-dimensional Ginzburg dg algebra $\Pi_n(Q, W)$ and an idempotent $e_0$.
\end{thmx}

Then we give examples to illustrate the construction of higher McKay quivers with potential in various dimensions, \cf Examples~\ref{ex:dimension 2}, \ref{ex:dimension 3}, and \ref{ex:dimension 4}. In Proposition~\ref{prop:SLn minimal model}, we prove that up to a Morita equivalence, the skew group algebra is quasi-isomorphic to the $n$-dimensional Ginzburg dg algebra associated with a higher McKay quiver with potential. This is the key to proving Theorem~\ref{thm:B}.

\subsection*{Acknowledgments}

The author is very grateful to Dong Yang for stimulating discussions and help with the references. The author thanks Bernhard Keller for helpful comments and suggestions, and Martin Kalck for help with the references. The author is indebted to an anonymous referee for a very careful reading of the manuscript and many helpful comments and suggestions.

\section{Notation}

The following notation is used throughout the article: We let $k$ be a field. Algebras have units and morphisms of algebras preserve the units. Modules are unital right modules. We assume that $k$ acts centrally on all bimodules we consider. For a $k$-algebra $A$, we denote the category of $A$-modules, finitely generated $A$-modules, $A$-bimodules, finitely generated $A$-bimodules by $\Mod A$, $\mod A$, $\Bimod A$, $\bimod A$, respectively. The degree of a homogeneous element $a$ in a graded vector space is denoted by $|a|$. We denote the shift functor of graded vector spaces by $\Si$ and write $\si \colon A\to \Si A$ for the canonical map of degree $-1$. We use cohomological grading so that differentials are of degree $1$. For a graded bimodule $V$, we write $TV$ for the graded tensor algebra $\bigoplus_{p\geq 0}V^{\ten p}$.

\section{Preliminaries} \label{section:preliminaries}

\subsection{Quivers with potential} \label{ss:quivers with potential}

Ginzburg \cite{Ginzburg06} introduced a particular class of Calabi--Yau dg algebras as follows, \cf also \cite{VandenBergh15}. Let $n$ be a positive integer and $Q=(Q_0,Q_1,s,t)$ a finite graded quiver whose arrows are concentrated in degrees $[\frac{2-n}{2},0]$. We write $e_i$ for the lazy
path at a vertex $i$ of $Q$.
Let $\eps \colon Q_1 \to \{ \pm 1\}$ be a map.
Let $\tilde{Q}$ be the graded quiver obtained from $Q$ by adding an arrow $\alpha^*\colon j \to i$ of degree $2-n-|\alpha|$ for each arrow $\alpha\colon i \to j$. The graded vector space $\Tr(k\tilde{Q})=k\tilde{Q}/[k\tilde{Q},k\tilde{Q}]$ is spaned by the cyclic equivalence classes of cycles in $\tilde{Q}$. Here the cyclic equivalence relation is generated by the $pp'=(-1)^{|p||p'|}p'p$, $p$, $p' \in k\tilde{Q}$. The necklace bracket $\{?, -\}_\eps$ on $\Tr(k\tilde{Q})$ is determined by
\begin{align*}
\{p, p'\}_\eps= & \sum_{\alpha \in Q_1}(\sum_{\{(u, v, u', v')\mid p=u\alpha v, p'=u'\alpha^* v'\}}(-1)^{|u'|+|\alpha|+|u|(|\alpha|+|v|)}\eps(\alpha)u'vuv' \\
 & -\sum_{\{(u, v, u', v')\mid p=u\alpha^* v, p'=u'\alpha v'\}}(-1)^{|u'|+(|\alpha|+1)|\alpha^*|+|u|(|\alpha|+|v|)}\eps(\alpha)u'vuv')\: ,
\end{align*}
where $u$, $v$, $u'$, $v'$ run through the paths in $\tilde{Q}$.
A {\em potential} $W$ on $\tilde{Q}$ is an element of $\Tr(k\tilde{Q})$ which is of degree $3-n$ and only contains paths of length at least $3$ such that we have $\{W, W\}_\eps=0$. For each arrow $\alpha$ of $\tilde{Q}$, we have the {\em cyclic derivative} $\del_{\alpha}\colon \Tr(k\tilde{Q})\to k\tilde{Q}$ which maps the class of a cycle $p$ to
\[
\sum_{\{(u,v)\mid p=u\alpha v\}}(-1)^{|u|(|\alpha|+|v|)}vu\: ,
\]
where $u$ and $v$ run through the paths in $\tilde{Q}$. Let $\bar{Q}$ be the graded quiver obtained from $\tilde{Q}$ by adding a loop $t_i$ of degree $1-n$ at each vertex $i$. The {\em $n$-dimensional (non-completed) Ginzburg dg algebra} $\Pi_n(Q, W, \eps)$ associated with the triple $(Q,W,\eps)$ is the dg path algebra $(k\bar{Q}, d)$ with the differential determined by
\begin{align*}
d(\alpha) & =-(-1)^{(|\alpha|+1)|\alpha^*|}\eps(\alpha)\del_{\alpha^*}W\: , \\
d(\alpha^*) & =(-1)^{|\alpha|}\eps(\alpha)\del_{\alpha}W \: , \\
\mbox{and }d(t_i) & =\sum_{\alpha \in Q_1} \eps(\alpha)e_i(\alpha \alpha^*-(-1)^{|\alpha||\alpha^*|}\alpha^* \alpha)e_i \: .
\end{align*}
We have $d^2(\alpha)=0$ for all $\alpha \in \tilde{Q}_1$ if and only if we have $\{W, W\}_\eps=0$. If the potential $W$ vanishes or the map $\eps$ is the constant map of value $1$, then we omit them respectively for simplicity.

Ginzburg dg algebras arise as examples of deformed Calabi--Yau completions in the sense of Keller \cite{Keller11b}. Notice that our definition is a special case of Van den Bergh's deformed dg preprojective algebras. If we let $Q^\dagger$ be the graded quiver obtained from $Q$ by adding, at each vertex, a loop of degree $\frac{2-n}{2}$ for each loop of odd degree $\frac{2-n}{2}$ at that vertex, then our $\Pi_n(Q, W)$ coincides with Van den Bergh's $\Pi(Q^\dagger, n, W)$.

\subsection{Singularity categories and cluster categories} \label{ss:singularity categories}

Let $R$ be a commutative noetherian ring. It is {\em regular} if for each prime ideal $\mathfrak{p}$ of $R$, the dimension of $\mathfrak{p}R_{\mathfrak{p}}/\mathfrak{p}^2 R_{\mathfrak{p}}$ over $R_{\mathfrak{p}}/\mathfrak{p}R_{\mathfrak{p}}$ equals the Krull dimension of $R_{\mathfrak{p}}$. Equivalently, the localization of $R$ at each prime ideal has finite global dimension. It is {\em Gorenstein} if its localization at each prime ideal has finite injective dimension as a module over itself. It has {\em isolated singularities} if its localization at each non-maximal prime ideal is regular. A finitely generated $R$-module $M$ is {\em (maximal) Cohen--Macaulay} if either $M$ vanishes, or for each prime ideal $\mathfrak{p}$ of $R$, the depth of the $R_{\mathfrak{p}}$-module $M_{\mathfrak{p}}$ equals the Krull dimension of $R_{\mathfrak{p}}$. We write $\CM R$ for the category of Cohen--Macaulay $R$-modules. Let $\cd^b(\mod R)$ be the bounded derived category of finitely generated $R$-modules. Its thick subcategory generated by the free $R$-module of rank one is the {\em perfect derived category} $\per R$. The {\em singularity category} $\sg R$ associated with $R$ is defined \cite{Buchweitz21, Orlov04} as the Verdier quotient $\cd^b(\mod R)/\per R$. A commutative noetherian ring $R$ is regular if and only if the singularity category $\sg R$ vanishes.

Let $A$ be a dg $k$-algebra. The thick subcategory of $\cd(A)$ generated by the free dg \linebreak $A$-module of rank one is the {\em perfect derived category} $\per A$. It consists of compact objects in $\cd(A)$. If the dg algebra $A$ is concentrated in degree $0$, then $\per A$ is triangle equivalent to the bounded homotopy category $\ch^b(\proj A)$ of finitely generated projective $A$-modules.
The {\em finite-dimensional derived category} $\cd_{fd}(A)$ is the full subcategory of the dg $A$-modules in $\cd(A)$ whose homology is of finite total dimension.
The {\em finitely generated derived category} $\cd_{fg}(A)$ is defined \cite{KalckYang18} as the full subcategory of the dg $A$-modules in $\cd(A)$ whose homology is finitely generated as a graded $H^0(A)$-module. It contains the subcategory $\cd_{fd}(A)$. If the ring $H^0(A)$ is right noetherian, then $\cd_{fg}(A)$ is a triangulated subcategory of $\cd(A)$. If the algebra $H^0(A)$ is finite-dimensional, then we have $\cd_{fg}(A)=\cd_{fd}(A)$. If the dg algebra $A$ is concentrated in degree $0$ and right noetherian as a ring, then $\cd_{fg}(A)$ is triangle equivalent to $\cd^b(\mod A)$. Recall that $A$ is {\em smooth} if $A$ is perfect as a dg $A$-bimodule. In this case, the subcategory $\cd_{fd}(A)$ is contained in $\per A$.

For an $n$-dimensional Ginzburg dg algebra $\Pi=\Pi_n(Q, W, \eps)$, the associated {\em cluster category} $\cc_{\Pi}$ is defined \cite{Amiot09, Plamondon11, Guo11} as the idempotent completion of the Verdier quotient $\per \Pi/\cd_{fd}(\Pi)$. For an $n$-dimensional Ginzburg dg algebra $\Pi$ such that the ring $H^0(\Pi)$ is right noetherian and the subcategory $\cd_{fg}(\Pi)$ is contained in $\per \Pi$, the associated {\em small cluster category} $\cc^s_{\Pi}$ is defined \cite{KalckYang18} as the idempotent completion of the Verdier quotient $\per \Pi/\cd_{fg}(\Pi)$. For a smooth dg algebra $A$, the associated {\em cosingularity category} $\cosg A$ is defined \cite{Keller24} as the Verdier quotient $\per A/\cd_{fd}(A)$. For a dg algebra $A$ such that the ring $H^0(A)$ is right noetherian and the subcategory $\cd_{fg}(A)$ is contained in $\per A$, the associated {\em small cosingularity category} $\cosg\!^s A$ is defined as the Verdier quotient $\per A/\cd_{fg}(A)$.

For an $n$-dimensional Ginzburg dg algebra $\Pi=\Pi_n(Q, W, \eps)$ and a (strict) idempotent $e$ (\ie $e^2=e$) of it, the associated {\em relative cluster category} $\cc_{\Pi, e}$ is defined \cite{Wu23a, KellerWu23} as the idempotent completion of the Verdier quotient $\per \Pi/\cd_{fd, e}(\Pi)$, where $\cd_{fd, e}(\Pi)$ denotes the triangulated subcategory of $\cd_{fd}(\Pi)$ generated by the dg $A$-modules annihilated by $e$. If the idempotent $e$ vanishes, then we recover the absolute notion. The associated {\em Higgs category} $\ch_{\Pi, e}$ is defined \cite{Wu23a, KellerWu23} as the full subcategory of $\cc_{\Pi, e}$ whose objects are the $X$ such that the morphism spaces $\cc_{\Pi, e}(\Pi, \Si^i X)$ are finite-dimensional for all $i=1$, \ldots, $n-2$, and we have $\cc_{\Pi, e}(e\Pi, \Si^i X)=0=\cc_{\Pi, e}(X, \Si^i e\Pi)$ for all positive integers $i$. If the subcategory $\add e\Pi \subseteq \add \Pi$ is functorially finite, the algebra $H^0(\Pi/e)$ is finite-dimensional, and the homology of $\Pi$ is concentrated in degree $0$, then $\ch_{\Pi, e}$ is a Frobenius exact category and its stable category is triangle equivalent to the cluster category associated with the dg quotient $\Pi/e$, \cf parts~(2) and (3) of Corollary~3.26 of \cite{Wu25}.

\section{The structure theorems}

\subsection{Quivers with potential for polynomial algebras}

Let $n$ be a positive integer. Denote the polynomial algebra $k[x_1, \ldots, x_n]$ by $R$. Let $Q$ be the graded quiver with a single vertex and arrows $\{x_S\}$, where $S$ runs through the nonempty subsets of $[n]=\{1, \ldots, n\}$ satisfying $|S|<\frac{n}{2}$, or $|S|=\frac{n}{2}$ and $1\in S$, and the arrow $x_S$ is of degrees $1-|S|$. Let $\tilde{Q}$ be the graded quiver obtained from $Q$ by adding an arrow $x_{S^c}$ of degree $1-|S^c|$ for each arrow $x_S$, where $S^c$ denotes the complement of $S$ in $[n]$. Let $\bar{Q}$ be the graded quiver obtained from $\tilde{Q}$ by adding a loop $x_{[n]}$ of degree $1-n$. We write $\wedge_{i\in S}x_i$ for the wedge product whose indices of the variables are in ascending order. For a disjoint union $A_1\sqcup \cdots \sqcup A_m=[n]$, we define $\eps_{A_1, \ldots, A_m}$ by
\[
(\wedge_{i\in A_1}x_i)\wedge \cdots \wedge(\wedge_{i\in A_m}x_i)=\eps_{A_1, \ldots , A_m}\wedge_{i\in [n]}x_i \: .
\]
Let $W$ be the potential
\[
\sum_{A\sqcup B\sqcup C=[n]}(-1)^{|B|-1}\eps_{A, B, C}x_A x_B x_C
\]
on $\tilde{Q}$, where $(A, B, C)$ runs through the triples of nonempty sets modulo the cyclic permutations. Let $\eps$ be the map $Q_1 \to \{ \pm 1\}$ which maps $x_S$ to $(-1)^{|S|-1}\eps_{S,S^c}$. In the spirit of Section~6.2 of \cite{ThanhofferVandenBergh16}, we construct quivers with potential for polynomial algebras as follows.

\begin{proposition} \label{prop:minimal model}
The morphism
\[
\Pi_n(Q, W, \eps)\longrightarrow R
\]
of dg algebras, which maps $x_{\{i\}}$ to $x_i$ for all $i\in [n]$ and $x_S$ to $0$ for all $S$ containing at least two elements, is a quasi-isomorphism.
\end{proposition}

\begin{proof}
By Section~6.2 of \cite{ThanhofferVandenBergh16}, there is a quasi-isomorphism
\[
(k\langle(x_S)_{S\neq \emptyset}\rangle, d)\longrightarrow R
\]
of dg algebras. The statement follows from the fact that the differential of $\Pi_n(Q, W, \eps)$ coincides with that of $(k\langle(x_S)_{S\neq \emptyset}\rangle, d)$.
\end{proof}

\subsection{A monoidal functor}

Let $G$ be a finite group. From now on, we always assume that the field $k$ is algebraically closed and the order of $G$ is not divisible by the characteristic of the ground field $k$. For two $kG$-modules $M$ and $N$, the $kG$-module $\Hom(M, N)$ is defined to be the vector space $\Hom_k(M, N)$ with the module structure
\[
(\varphi g)(m)=\varphi(mg^{-1})g
\]
for all $\varphi \in \Hom_k(M, N)$, $g\in G$, and $m\in M$. For two $kG$-modules $M$ and $N$, the \linebreak $kG$-module $M\ten N$ is defined to be the vector space $M\ten_k N$ with the module structure
\[
(m\ten n)g=mg\ten ng
\]
for all $m\in M$, $n\in N$, and $g\in G$. With these definitions, one easily checks that we have the canonical isomorphisms
\[
\begin{tikzcd}
\Hom_{kG}(X\ten Y, Z) \arrow[no head]{r}{\sim} & \Hom_{kG}(X, \Hom(Y, Z))\ko X, Y, Z\in \Mod kG \: ,
\end{tikzcd}
\]
mapping a morphism $f\colon X\ten Y\to Z$ to the morphism $x\mapsto f(x\ten ?)$. In other words, the internal $\Hom$-functor $\Hom(?, -)$ and the tensor functor $?\ten -$ make the category $\Mod kG$ into a closed monoidal category. If we let the second argument of the internal $\Hom$-functor be the trivial $kG$-module $k$, then it restricts to the duality
\[
* \colon (\mod kG)\op \iso \mod kG \: .
\]
Let $\{L_0, \ldots, L_m\}$ be a complete set of irreducible representations of $G$ with the trivial representation $L_0$. Put $I=\{0, \ldots, m\}$ and denote by $kI$ the product of copies indexed by $I$ of the ground field $k$. The category $\Bimod kI$ also admits a closed monoidal structure with the internal $\Hom$-functor $\Hom(?, -)=\Hom_{kI}(?, -)$ and the tensor functor $?\ten -=\, ?\ten_{kI}-$. If we let the second argument of the internal $\Hom$-functor be the $kI$-bimodule $kI$, then it restricts to the duality
\[
* \colon (\mod kI)\op \iso \mod kI\: .
\]
By Maschke's theorem, the group algebra $kG$ is semisimple. We may and will assume that we have $kG=\prod_{i\in I}\End_{kG}(L_i^{\oplus n_i})$. For any $i\in I$, let $e_i$ be a primitive idempotent in $\End_{kG}(L_i^{\oplus n_i})$. Put $e=\sum_{i\in I}e_i$ and $L=\bigoplus_{i\in I}L_i$. In the spirit of the proof of Theorem~4.4.6 of \cite{Ginzburg06}, we define $F$ to be the functor $\Hom_{kG}(L, L\ten ?)\colon \Mod kG \to \Bimod kI$. For an exact category $\ca$, we write $\cc(\ca)$ for the category of complexes in $\ca$. We prove the following lemmas.

\begin{lemma} \label{lemma:functor property}
The functor $F$ is exact and monoidal, and so is the induced functor
\[
F\colon \cc(\Mod kG) \longrightarrow \cc(\Bimod kI)
\]
between categories of complexes. The restriction $\mod kG \to \bimod kI$ of $F$ is compatible with the dualities on both sides.
\end{lemma}

\begin{proof}
Since the $kG$-module $L$ is projective, the functor $\Hom_{kG}(L, L\ten ?)$ is exact. Since the algebra $kG$ is semisimple, we have the canonical decomposition
\[
L_i \ten X=\bigoplus_{j\in I}L_{ij}^X
\]
as a $kG$-module for all $i\in I$ and $kG$-modules $X$, where $L_{ij}^X$ is isomorphic to a direct sum of copies of $L_j$. Let us prove that we have an isomorphism
\[
\begin{tikzcd}
\Hom_{kG}(L, L\ten M)\ten \Hom_{kG}(L, L\ten N) \arrow{r}{\sim} & \Hom_{kG}(L, L\ten M\ten N)
\end{tikzcd}
\]
of $kI$-bimodules for all $kG$-modules $M$ and $N$. It suffices to prove that we have
\[
\begin{tikzcd}
\Hom_{kG}(L, L_j\ten M)\ten \Hom_{kG}(L_i, L\ten N) \arrow{r}{\sim} & \Hom_{kG}(L_i, L_j\ten M\ten N)
\end{tikzcd}
\]
for all $i$, $j\in I$. The left hand side equals
\begin{align*}
 & \bigoplus_{k\in I}\Hom_{kG}(L_k, L_j\ten M)\ten \Hom_{kG}(L_i, L_k\ten N) \\
= & \bigoplus_{k\in I}\Hom_{kG}(L_k, L_{jk}^M)\ten \Hom_{kG}(L_i, L_{ki}^N) \: .
\end{align*}
The right hand side is isomorphic to
\[
\bigoplus_{k\in I}\Hom_{kG}(L_i, L_{jk}^M \ten N) \: .
\]
Since the field $k$ is algebraically closed, we have the canonical isomorphism
\[
\Hom_{kG}(L_k, L_{jk}^M)\ten_k L_k \xlongrightarrow{_\sim} L_{jk}^M
\]
of $kG$-modules. It follows that we have
\[
\begin{tikzcd}
\bigoplus_{k\in I}\Hom_{kG}(L_i, L_{jk}^M \ten N) & \bigoplus_{k\in I}\Hom_{kG}(L_k, L_{jk}^M)\ten \Hom_{kG}(L_i, L_{ki}^N)\arrow[swap]{l}{\sim} \: .
\end{tikzcd}
\]
We conclude that the functor $F$ is monoidal. Clearly, the induced functor
\[
F\colon \cc(\Mod kG) \longrightarrow \cc(\Bimod kI)
\]
also has the same properties. We now prove the last statement. Let $M$ be a finite-dimensional $kG$-module. We have the canonical isomorphism
\[
L\ten M^* \xlongrightarrow{_\sim} \Hom(M, L)
\]
of $kG$-modules. By the adjunction $(?\ten M, \Hom(M, ?))$ in the closed monoidal category $\Mod kG$, we have the isomorphism
\[
\begin{tikzcd}
\Hom_{kG}(L\ten M, L) \arrow[no head]{r}{\sim} & \Hom_{kG}(L, \Hom(M, L)) \: .
\end{tikzcd}
\]
Since the canonical bilinear form
\[
\begin{tikzcd}
\Hom_{kG}(L\ten M, L)\ten \Hom_{kG}(L, L\ten M)\arrow{r} & \Hom_{kG}(L, L)\arrow{r} & k
\end{tikzcd}
\]
is non-degenerate, we have the isomorphism
\[
\begin{tikzcd}
\Hom_{kG}(L\ten M, L)\arrow{r}{\sim} & \Hom_{kG}(L, L\ten M)^*\: .
\end{tikzcd}
\]
Combining the above isomorphisms we conclude that the restriction of $F$ to $\mod kG$ is compatible with the dualities.
\end{proof}

For a $kG$-module $M$, we define the $kG$-bimodule $G\#M$ to be the tensor space $kG\ten_k M$ with the bimodule structure
\[
g'(g\ten m)g''=g'gg''\ten mg''
\]
for all $m\in M$, $g$, $g'$, and $g''\in G$. It gives rise to a functor $G\# ? \colon \Mod kG \to \Bimod kG$.

\begin{lemma}
For an algebra object $A$ in $\Mod kG$, the multiplication
\[
\mu \colon (G\#A)\ten(G\#A) \longrightarrow G\#A
\]
determined by $\mu(g\ten a, g'\ten a')=gg'\ten(ag')a'$ and the unit
\[
\eta \colon kG\to G\#A
\]
determined by $\eta(g)=g\ten 1_A$ make $G\#A$ into an algebra object in $\Bimod kG$.
\end{lemma}

\begin{proof}
It suffices to prove that $(G\#A, \mu, \eta)$ satisfies the associativity and the unit law. For any $g\ten a$, $g'\ten a'$, and $g''\ten a'' \in G\#A$, we have
\[
\mu(\mu(g\ten a, g'\ten a'), g''\ten a'') = gg'g''\ten(((ag')a')g'')a'' = gg'g''\ten(ag'g'')(a'g'')a''\: ,
\]
where the second equality follows from the fact that the multiplication $A\ten A\to A$ is a morphism of $kG$-modules. On the other hand, we have
\[
\mu(g\ten a, \mu(g'\ten a', g''\ten a'')) = gg'g''\ten (ag'g'')(a'g'')a'' \: .
\]
These imply the associativity. For any $g\ten a\in G\#A$, we have
\[
\mu(\eta(1_G), g\ten a) = g\ten (1_A \, g)a = g\ten a \: ,
\]
where the second equality follows from the fact that the unit $k\to A$ is a morphism of $kG$-modules. On the other hand, we have
\[
\mu(g\ten a, \eta(1_G)) = g\ten (a\, 1_G)1_A = g\ten a \: .
\]
These imply the unit law.
\end{proof}

\begin{lemma} \label{lemma:natural isomorphism}
We have a natural isomorphism $F\simeq e(G\# ?)e$. Moreover, for any algebra object $A$ in $\Mod kG$, the isomorphism $FA\simeq e(G\# A)e$ is compatible with the algebra object structures on both sides.
\end{lemma}

\begin{proof}
Since the $kG$-module $L$ is projective, the first statement follows from the natural isomorphisms
\[
\Hom_{kG}(L, L\ten ?) \xlongleftarrow{_\sim} \Hom_{kG}(L, L\ten_{kG}(G\# ?)) \xlongrightarrow{_\sim} \Hom_{kG}(L, e(G\# ?))\xlongrightarrow{_\sim} e(G\# ?)e\: .
\]
The second statement follows from the definitions of the algebra object structures on both sides.
\end{proof}

\subsection{The structure theorem in the $\GL_n$ case}

Recall that $R$ denotes the polynomial algebra $k[x_1, \ldots, x_n]$. Denote the power series algebra $k\llbracket x_1, \ldots, x_n \rrbracket$ by $\hat{R}$. From now on, we always assume that $G$ is a finite subgroup of $\GL_n$. Let $V$ be the natural representation of $G$ given by the inclusion $G\to \GL_n$. It gives rise to a group action of $G$ on the polynomial algebra $R=\Sym V^*$. Denote the $G$-invariant subalgebra of $R$ by $R^G$. Let $U$ be the graded $kG$-module $\bigoplus_{p=1}^n \Si^{-p}V^{\wedge p}$ and $U_c$ its graded $kG$-submodule $\bigoplus_{p=1}^{n-1} \Si^{-p}V^{\wedge p}$. Let $\wedge$ be the morphism $U_c \ten U_c \to U$ which maps $u\ten u'$ to $u\wedge u'$. Let $A$ be the dg tensor algebra $(T(F\Si^{-1}U^*), d)$ with the differential determined by the composed map
\[
-(\si^{-1}\ten \si^{-1})\circ (F\wedge)^*\circ \si \colon F\Si^{-1}U^* \longrightarrow F\Si^{-1}U_c^* \ten F\Si^{-1}U_c^* \: .
\]
Let $A'$ be the dg quotient $A/e_0$. We describe the singularity category $\sg R^G$ in the $\GL_n$ case as follows.

\begin{theorem} \label{thm:GLn}
Let $n$ be a positive integer. Let $k$ be an algebraically closed field and $G\subset \GL_n$ a finite subgroup whose order is not divisible by the characteristic of $k$.
\begin{itemize}
\item[a)] There exists a triangle equivalence up to direct summands
\[
\sg R^G \to \cosg\!^s A'
\]
mapping $(L\ten R)^G$ to $A'$.
\item[b)] If $R^G$ has isolated singularities, then there exist triangle equivalences
\[
\begin{tikzcd}
\sg \hat{R}^G & \sg R^G \arrow[swap]{l}{\sim} \arrow{r}{\sim} & \cosg\!^s A' \: .
\end{tikzcd}
\]
\end{itemize}
\end{theorem}

\begin{lemma} \label{lemma:invariant subalgebra}
The algebra $R^G$ is isomorphic to $e_0(G\# R)e_0$. Under the identification between modules over $R^G$ and $e_0(G\# R)e_0$ via this isomorphism, the $R^G$-module $(L_i \ten R)^G$ is isomorphic to $e_i(G\# R)e_0$ for all $i\in I$.
\end{lemma}

\begin{proof}
By Lemma~\ref{lemma:natural isomorphism}, both the statements follow from the canonical isomorphisms
\[
\begin{tikzcd}
(L_i \ten R)^G & \Hom_{kG}(L_0, L_i\ten R)\arrow[swap]{l}{\sim}\arrow[no head]{r}{\sim} & e_i(G\# R)e_0
\end{tikzcd}
\]
for all $i\in I$.
\end{proof}

\begin{proposition} \label{prop:GLn minimal model}
There is a quasi-isomorphism
\[
A \to e(G\# R)e
\]
of dg algebras.
\end{proposition}

\begin{proof}
From Proposition~\ref{prop:minimal model}, we deduce that we have the quasi-isomorphism
\[
(T(\Si^{-1}U^*), d)\longrightarrow \Sym V^*
\]
between algebra objects in the monoidal category $\cc(\Mod kG)$. By Lemmas~\ref{lemma:functor property} and \ref{lemma:natural isomorphism}, its image
\[
(T(F\Si^{-1}U^*), d) \longrightarrow e(G\# R)e
\]
under the monoidal functor $F$ is also a quasi-isomorphism between algebra objects in the monoidal category $\cc(\Bimod kI)$. By definition, the differential of $\Si^{-1}U^*$ in $(T(\Si^{-1}U^*), d)$ is given by the composed map
\[
-(\si^{-1} \ten \si^{-1})\circ \wedge^* \circ \si \colon \Si^{-1}U^* \longrightarrow \Si^{-1}U_c^* \ten \Si^{-1}U_c^* \: .
\]
By Lemma~\ref{lemma:functor property}, the differential of $F\Si^{-1}U^*$ in $(T(F\Si^{-1}U^*), d)$ is given by the composed map
\[
-(\si^{-1}\ten \si^{-1})\circ (F\wedge)^*\circ \si \colon F\Si^{-1}U^* \longrightarrow F\Si^{-1}U_c^* \ten F\Si^{-1}U_c^* \: .
\]
This means that the dg algebra $(T(F\Si^{-1}U^*), d)$ equals $A$ and the statement follows.
\end{proof}

\begin{proof}[Proof of Theorem~\ref{thm:GLn}]
a) Since the ring $R$ is right noetherian and the group $G$ is finite, by Lemma~1.5.11 of \cite{McConnellRobson01}, the ring $G\# R$ is right noetherian. By Morita equivalence, the ring $e(G\# R)e$ is also right noetherian. Since $F\Si^{-1}U^*$ is of finite total dimension, the dg algebra $A$ is smooth. By Proposition~\ref{prop:GLn minimal model}, the algebra $e(G\# R)e$ is of finite global dimension. By Proposition~\ref{prop:GLn minimal model}, Lemma~\ref{lemma:invariant subalgebra}, and Theorem~6.6 of \cite{KalckYang18}, there is a triangle equivalence up to direct summand $\sg R^G \to \cosg\!^s B$ mapping $(L\ten R)^G$ to $B$ for some dg algebra $B$. To prove that the dg algebra $B$ coincides with $A'$, it suffices to prove that $B$ is the dg quotient of $A$ by $e_0$. This follows from Theorem~7.1 of \cite{KalckYang18}, \cf also Theorem~3.3.1 of \cite{Booth21}.

b) Since the group $G$ is finite, by Hilbert's finiteness theorem, the algebra $R^G$ is finitely generated. Since $R^G$ has isolated singularities, the first triangle equivalence follows from Proposition~A.8 of \cite{KellerMurfetVandenBergh11}. By Corollary~2.3 of \cite{PavicShinder21}, the category $\sg R^G$ is Karoubian. Then the second triangle equivalence follows from part~a) and Theorem~6.6 of \cite{KalckYang18}, \cf also Proposition~6.2.12 of \cite{Booth21}.
\end{proof}

\subsection{The structure theorem in the $\SL_n$ case}

Recall that $R$ denotes the polynomial algebra $k[x_1, \ldots, x_n]$ and $\hat{R}$ denotes the power series algebra $k\llbracket x_1, \ldots, x_n \rrbracket$. From now on, we always assume that $G$ is a finite subgroup of $\SL_n$. We define a {\em higher McKay quiver with potential} $(Q, W)$ associated with the finite subgroup $G\subseteq \SL_n$ as follows, \cf Section~4.4 of \cite{Ginzburg06}, Section~3 of \cite{BocklandtSchedlerWemyss10}, Section~5 of \cite{IyamaTakahashi13}, and Section~5.6 of \cite{Lam14} for the previous work.

Put $Q_0=I$. Let $p$ be the composition of $\si^n$ with the canonical projection $U\to \Si^{-n}V^{\wedge n}$. Since $G$ is a subgroup of $\SL_n$, the $kG$-module $V^{\wedge n}$ is isomorphic to $L_0$. Fix an isomorphism $\pi \colon V^{\wedge n}\iso k$. Since the bilinear form
\[
\langle -, -\rangle = \pi \circ p \circ \wedge \colon U_c \ten U_c \to k
\]
is non-degenerate and graded symmetric of degree $-n$, the bilinear form
\[
\langle -, -\rangle \circ (\si^{-1}\ten \si^{-1}) \colon \Si U_c \ten \Si U_c \to k
\]
is non-degenerate and graded anti-symmetric and so is its image under the functor $F$. By Lemma~\ref{lemma:functor property}, taking the dual we deduce that the image of $1$ under the map
\[
(\si^{-1}\ten \si^{-1})\circ F\langle -, -\rangle^* \colon kI \longrightarrow F\Si^{-1}U_c^* \ten F\Si^{-1}U_c^*
\]
is non-degenerate and graded anti-symmetric. Let $Q_1$ be a suitable homogeneous $k$-basis of a Lagrangian homogeneous $kI$-subbimodule of $F\Si^{-1}U_c^*$ concentrated in degrees $[\frac{2-n}{2}, 0]$ such that we can write this image as the sum $\sum_{\alpha \in Q_1}[\alpha, \alpha^*]$. Put $\tilde{Q}_1=Q_1\cup \{\alpha^* \mid \alpha \in Q_1\}$. Let $(\alpha^\vee)_{\alpha \in \tilde{Q}_1}$ be the $k$-basis of $F\Si U_c$ which is graded dual to the $k$-basis $(\alpha)_{\alpha \in \tilde{Q}_1}$ of $F\Si^{-1}U_c^*$. Let $\wedge^2$ be the morphism
\[
U_c \ten U_c \ten U_c \longrightarrow U
\]
which maps $u\ten u'\ten u''$ to $u\wedge u'\wedge u''$. Denote the map $\pi \circ p \circ \wedge^2$ by $\langle -, -, -\rangle$. Notice that for any cycle $\alpha \beta \gamma$ in $\tilde{Q}_1$, we can view
\[
\Phi=(\si^{-1}\gamma^\vee) \circ (\si^{-1}\beta^\vee) \circ (\si^{-1}\alpha^\vee)
\]
as a composed morphism in $\cc(\mod kG)$ from $L_i$ to $L_i\ten U_c \ten U_c \ten U_c$, where $i$ denotes the source of $\alpha^\vee$, \cf~the proof of Lemma~\ref{lemma:functor property}. It follows that the composition $F\langle -, -, -\rangle \circ \Phi$ is an endomorphism of the $kG$-module $L_i$. Since the field $k$ is algebraically closed, it is the multiplication by a scalar $\lambda_{\alpha \beta \gamma}\in k$. We define $W\in \Tr(k\tilde{Q})$ as the sum
\[
\sum_{\alpha, \beta, \gamma \in \tilde{Q}_1}(-1)^{|\beta|}\lambda_{\alpha \beta \gamma}\alpha \beta \gamma \: ,
\]
where $\alpha \beta \gamma$ runs through the cyclic equivalence classes up to a minus sign of cycles in $\tilde{Q}$. Recall that the proof of Proposition~\ref{prop:GLn minimal model} shows that $A$ equals the image of the dg algebra $(T(\Si^{-1}U^*), d)$ under $F$. In particular, the differential of $F\Si^{-1}U_c^*$ in $A$ squares to zero. By the proof of Proposition~\ref{prop:SLn minimal model}, so does the differential of $\tilde{Q}_1$ in $\Pi_n(Q, W)$. This means that we have $\{W, W\}=0$.

\begin{definition} \label{def:McKay quiver}
We call $(Q, W)$ defined above a higher McKay quiver with potential associated with the finite subgroup $G\subseteq \SL_n$.
\end{definition}

\begin{remarks} \label{rk:well-defined} \mbox{}
\begin{itemize}
\item[a)] The number of arrows of degree $1-p$ from $i$ to $j$ in $\bar{Q}$ equals the multiplicity of the direct summand $L_i$ in the decomposition of the $kG$-module $L_j\ten (V^{\wedge p})^*$ for all positive integers $p$.
\item[b)] Higher McKay quivers with potential $(Q, W)$ depend on not only the finite group $G$ but also the embedding $G\subseteq \SL_n$.
\item[c)] Our quivers $Q$ depend on various choices of the Lagrangian homogeneous \linebreak $kI$-subbimodule of $F\Si^{-1}U_c^*$, and our potentials $W$ depend on various choices of the isomorphism $\pi$ and of the $k$-basis of $F\Si^{-1}U_c^*$. Nevertheless, by the proof of Proposition~\ref{prop:SLn minimal model}, the associated $n$-dimensional Ginzburg dg algebras are isomorphic for variant choices.
\end{itemize}
\end{remarks}

Let $Q'$ be the graded quiver obtained from $Q$ by removing the vertex $0$ and the arrows which are adjacent to it. Let $W'$ be the image of $W$ under the canonical surjection $\Tr(k\tilde{Q})\to \Tr(k\tilde{Q'})$. We describe the singularity category $\sg R^G$ in the $\SL_n$ case as follows.

\begin{theorem} \label{thm:SLn}
Let $n$ be a positive integer. Let $k$ be an algebraically closed field and $G\subseteq \SL_n$ a finite subgroup whose order is not divisible by the characteristic of $k$.
\begin{itemize}
\item[a)] There exists a triangle equivalence up to direct summands
\[
\sg R^G \longrightarrow \cc^s_{\Pi_n(Q', W')}
\]
mapping $(L\ten R)^G$ to $\Pi_n(Q', W')$.
\item[b)] If $R^G$ has isolated singularities, then there exists an equivalence
\[
\CM R^G \xlongrightarrow{_\sim} \ch_{\Pi_n(Q, W),e_0}
\]
of exact categories mapping $(L\ten R)^G$ to $\Pi_n(Q, W)$. It induces a triangle equivalence
\[
\sg R^G \xlongrightarrow{_\sim} \cc_{\Pi_n(Q', W')}
\]
between $(n-1)$-Calabi--Yau triangulated categories. If moreover, we have $n\geq 2$, then both categories $\sg R^G$ and $\CM R^G$ contain the canonical $(n-1)$-cluster-tilting object $(L\ten R)^G$.
\end{itemize}
\end{theorem}

The following examples serve to illustrate the construction of higher McKay quivers with potential in dimension $2$, $3$, and $4$, respectively.

\begin{example} \label{ex:dimension 2}
Recall that McKay \cite{McKay83} gives the following correspondence between Kleinian singularities and finite subgroups of $\SL_2$, \cf also Section~6.3 of \cite{LeuschkeWiegand12}.
\[
\begin{array}{|c||c|}
\hline
\text{singularity} & \text{subgroup} \\
\hline
A_m & C_{m+1}, \text{cyclic} \\
\hline
D_m & D_{m-2}, \text{binary dihedral} \\
\hline
E_6 & T, \text{binary tetrahedral} \\
\hline
E_7 & O, \text{binary octahedral} \\
\hline
E_8 & I, \text{binary icosahedral} \\
\hline
\end{array}
\]
In the case of type $A_m$, the finite subgroup $G\subset \SL_2$ is generated by
$\begin{bmatrix}
\omega & 0 \\
0 & \omega^{-1}
\end{bmatrix}$, where $\omega \in \C$ is a primitive $(m+1)$th root of unity. Then the higher McKay quiver $Q$ associated with $G\subset \SL_2$ is the extended Dynkin quiver $\tilde{A}_m$ with the set of vertices $\Z/(m+1)\Z$ and the arrows from $i$ to $i+1$, which recovers the McKay correspondence. The potential $W$ associated with $G\subset \SL_2$ vanishes for degree reason.

Let $Q'$ be the quiver with potential obtained from $Q$ by removing the vertex $0$ and the arrows which are adjacent to it. It gives rise to the $2$-dimensional Ginzburg dg algebra $\Pi_2(Q')$. In this case, the commutative ring $\C[x_1, x_2]^G$ has an isolated singularity. So the category $\CM \C[x_1, x_2]^G$ of Cohen--Macaulay modules is equivalent to the Higgs category $\ch_{\Pi_2(Q), e_0}$. The singularity category $\sg \C[x_1, x_2]^G$ is triangle equivalent to the cluster category $\cc_{\Pi_2(Q')}$.
\end{example}

\begin{example} \label{ex:dimension 3}
Let $k$ be the field $\C$ and $G\subset \SL_3$ the finite subgroup generated by
$\begin{bmatrix}
\omega & 0 & 0 \\
0 & \omega & 0 \\
0 & 0 & \omega
\end{bmatrix}$, where $\omega \in \C$ is a primitive third root of unity. Then the higher McKay quiver $Q$ associated with $G\subset \SL_3$ is the quiver
\[
\begin{tikzcd}[row sep=1.5cm, column sep=1.5cm]
	& 1\arrow[dr,"3" description] \\
	0\arrow[ur,"3" description] & &
	2\arrow[ll,"3" description] \mathrlap{\: ,}
\end{tikzcd}
\]
where the attached number of an arrow represents its multiplicity. Denote the three arrows with source $0$, $1$, and $2$ by $\alpha_i$, $\beta_i$, and $\gamma_i$, respectively, $i=1$, $2$, $3$. A potential $W$ associated with $G\subset \SL_3$ is $\sum_{\si\in \mathrm{S}_3}(-1)^{\mathrm{sgn}(\si)}\alpha_{\si(1)}\beta_{\si(2)}\gamma_{\si(3)}$. This recovers the McKay quiver with potential in dimension $3$, \cf~\cite{KalckYang18}.

Let $(Q', W')$ be the quiver with potential obtained from $(Q, W)$ by removing the vertex $0$ and the arrows which are adjacent to it. Clearly, the potential $W'$ vanishes. The quiver $Q'$ gives rise to the $3$-dimensional Ginzburg dg algebra $\Pi_3(Q')$. In this case, the commutative ring $\C[x_1, x_2, x_3]^G$ has an isolated singularity. Then the category $\CM \C[x_1, x_2, x_3]^G$ of Cohen--Macaulay modules is equivalent to the Higgs category $\ch_{\Pi_3(Q, W), e_0}$. The singularity category $\sg \C[x_1, x_2, x_3]^G$ is triangle equivalent to the cluster category $\cc_{\Pi_3(Q')}$.
\end{example}

We give the following example in the higher-dimensional and non-abelian case. Notice that the non-abelianness of the subgroup causes the asymmetry of the higher McKay quiver.

\begin{example} \label{ex:dimension 4}
Let $k$ be the field $\ol{\F_5}$ and $G\subset \SL_4$ the finite subgroup given by the image of the composed morphism
\[
\begin{tikzcd}
\mathrm{S}_3 \arrow{r} & \GL_3 \arrow{r} & \SL_4 \: ,
\end{tikzcd}
\]
where the first morphism is given by the natural representation of $S_3$ and the second morphism maps $g$ to
$\begin{bmatrix}
g & 0 \\
0 & \det(g)^{-1}
\end{bmatrix}$. Then a higher McKay quiver $Q$ associated with $G\subset \SL_4$ is the graded quiver
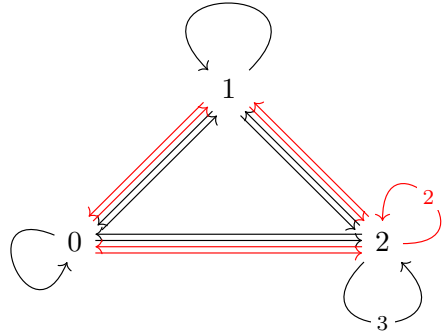
\[
\begin{tikzcd}[row sep=1.5cm, column sep=1.5cm]
	& 1\arrow[dr,shift right=0.75ex]
     \arrow[dr,shift right=-0.25ex,red]
	\arrow[dl,shift right=-0.25ex]
     \arrow[dl,shift right=0.75ex,red]
	\arrow[out=45,in=135,loop] & \\
	0\arrow[ur,shift right=0.75ex]
     \arrow[ur,shift right=-0.25ex,red]
	\arrow[rr,shift right=-0.25ex]
     \arrow[rr,shift right=0.75ex,red]
	\arrow[out=157.5,in=247.5,loop] & &
	2\arrow[ll,shift right=0.75ex]
     \arrow[ll,shift right=-0.25ex,red]
	\arrow[ul,shift right=-0.25ex]
     \arrow[ul,shift right=0.75ex,red]
	\arrow[out=225,in=315,loop,"3" description]
	\arrow[out=0,in=90,loop,red,"2" description]
\end{tikzcd}
\]
with the black arrows of degree $0$ and the red arrows of degree $-1$, where the attached number of a loop represents its multiplicity. Let $\tilde{Q}$ be the graded quiver obtained from $Q$ by adding an arrow $\alpha^*\colon j \to i$ of degree $-1$ (respectively, $-2$) for each arrow $\alpha\colon i \to j$ of degree $-1$ (respectively, $0$). A potential $W$ associated with $G\subset \SL_4$ is a linear combination of classes of cycles $\alpha \beta \gamma$ in $\tilde{Q}$ with $|\alpha|=|\beta|=0$ and $|\gamma|=-1$.

Let $(Q', W')$ be the quiver with potential obtained from $(Q, W)$ by removing the vertex $0$ and the arrows which are adjacent to it. It gives rise to the $4$-dimensional Ginzburg dg algebra $\Pi_4(Q', W')$. In this case, the commutative ring $\ol{\F_5}[x_1, x_2, x_3, x_4]^G$ does not have isolated singularities. So the singularity category $\sg \ol{\F_5}[x_1, x_2, x_3, x_4]^G$ is triangle equivalent up to direct summands to the small cluster category $\cc^s_{\Pi_4(Q', W')}$.
\end{example}

Let $\sym$ be the cyclic symmetrization map $\Tr(k\tilde{Q})\to k\tilde{Q}$ which vanishes on $k\tilde{Q}_0$ and maps the class of a cycle $\alpha_1 \ldots \alpha_p$ in $\tilde{Q}$ to the sum
\[
\sum_{1\leq i\leq p}\pm \alpha_i\ldots \alpha_p \alpha_1 \ldots \alpha_{i-1}\: ,
\]
where the sign is given by the Koszul sign rule.

\begin{lemma} \label{lem:potential}
We have $\sym(W)=((\si^{-1}\ten \si^{-1}\ten \si^{-1})\circ F\langle -, -, -\rangle^*)(1)$.
\end{lemma}

\begin{proof}
The right hand side can be written as a sum $\sum_{\alpha, \beta, \gamma \in \tilde{Q}_1}\mu_{\alpha \beta \gamma}\alpha \beta \gamma$, where $\alpha \beta \gamma$ runs through the cycles in $\tilde{Q}$. Since $(\si^{-1}\ten \si^{-1}\ten \si^{-1})\circ F\langle -, -, -\rangle^*$ is the dual map of $F\langle -, -, -\rangle \circ (\si^{-1}\ten \si^{-1}\ten \si^{-1})$, the coefficient $\mu_{\alpha \beta \gamma}$ is determined by the image of $\gamma^\vee \beta^\vee \alpha^\vee$ under the map $F\langle -, -, -\rangle \circ (\si^{-1}\ten \si^{-1}\ten \si^{-1})$, which is the multiplication by a scalar in $k$. By comparing this scalar $\mu_{\alpha \beta \gamma}$ with the coefficient $\lambda_{\alpha \beta \gamma}$ appearing in the definition of $W\in \Tr(k\tilde{Q})$ we obtain the desired equality.
\end{proof}

\begin{proposition} \label{prop:SLn minimal model}
There is a quasi-isomorphism
\[
\Pi_n(Q, W)\longrightarrow e(G\# R)e
\]
of dg algebras.
\end{proposition}

\begin{proof}
Define $\varphi \colon \Pi_n(Q, W)\to A$ to be the morphism of graded algebras which restricts to the identity on $F\Si^{-1}U_c^*$ and maps $t_i$ to $-\si^{n-1}(F\pi)^*(e_i)$. By Proposition~\ref{prop:GLn minimal model}, it suffices to prove that $\varphi$ is an isomorphism of dg algebras. Clearly, it is an isomorphism of graded algebras and hence it suffices to show that the differential of $F\Si^{-1}U^*$ coincides with that of $k\bar{Q}_1$. Recall that the differential of $F\Si^{-1}U^*$ is given by $-(\si^{-1}\ten \si^{-1})\circ (F\wedge)^*\circ \si$ and we have
\begin{equation} \label{eq:sum of 2-cycles}
((\si^{-1}\ten \si^{-1})\circ F\langle -, -\rangle^*)(1)=\sum_{\alpha \in Q_1}[\alpha, \alpha^*] \: .
\end{equation}
It follows that we have
\[
d(\varphi(t_i))=d(-\si^{n-1}(F\pi)^*(e_i))=\sum_{\alpha \in Q_1}e_i[\alpha, \alpha^*]e_i=\varphi(d(t_i))
\]
for all $i\in I$. Let $\phi$ be the isomorphism $U_c \to \Si^{-n}U_c^*$ induced by the non-degenerate bilinear form $\langle -, -\rangle$ and $\psi$ the morphism
\[
U_c \ten U_c \longrightarrow \Si^{-n}U_c^*
\]
induced by the map $\langle -, -, -\rangle$. Since the composition of the canonical projection \mbox{$U\to U_c$} with $\wedge$ equals $\phi^{-1}\circ \psi$, the restriction of $(F\wedge)^*$ to $FU_c^*$ equals $(F\psi)^* \circ ((F\phi)^*)^{-1}$. Therefore, the differential of $F\Si^{-1}U_c^*$ in $A$ is given by
\[
-(\si^{-1}\ten \si^{-1})\circ (F\wedge)^*\circ \si = -(\si^{-1}\ten \si^{-1})\circ (F\psi)^* \circ ((F\phi)^*)^{-1} \circ \si \: .
\]
Let $((\si \alpha)^\vee)_{\alpha \in \tilde{Q}_1}$ be the $k$-basis of $FU_c$ which is graded dual to the $k$-basis $(\si \alpha)_{\alpha \in \tilde{Q}_1}$ of $FU_c^*$. By the equality~(\ref{eq:sum of 2-cycles}), we have
\[
F\langle -, -\rangle^*(1)=\sum_{\alpha \in Q_1}((-1)^{|\alpha|}\si \alpha \ten \si \alpha^*-(-1)^{(|\alpha|+1)|\alpha^*|}\si \alpha^* \ten \si \alpha) \: .
\]
Since $\phi$ is induced by $\langle -, -\rangle$, it follows that we have
\[
(F\phi)^*((-1)^{|\alpha|}\si^n(\si \alpha)^\vee)=\si \alpha^* \: .
\]
By Lemma~\ref{lem:potential}, we have $F\langle -, -, -\rangle^*(1)=(\si \ten \si \ten \si)(\sym(W))$. Since $\psi$ is induced by $\langle -, -, -\rangle$, it follows that we have
\[
(F\psi)^*(\si^n(\si \alpha)^\vee)=((\si \alpha)^\vee \ten \id \ten \id)(\si \ten \si \ten \si)(\sym(W))=(\si \ten \si)\del_{\alpha}W \: .
\]
We conclude that we have $d(\alpha^*)=(-1)^{|\alpha|}\del_{\alpha}W$ for all $\alpha^* \in F\Si^{-1}U_c^*$. Similarly, from the equality
\[
(F\phi)^*(-(-1)^{(|\alpha|+1)|\alpha^*|}\si^n(\si \alpha^*)^\vee)=\si \alpha
\]
we deduce that we have $d(\alpha)=-(-1)^{(|\alpha|+1)|\alpha^*|}\del_{\alpha^*}W$ for all $\alpha \in F\Si^{-1}U_c^*$. This concludes the proof of that the differential of $F\Si^{-1}U^*$ coincides with that of $k\bar{Q}_1$.
\end{proof}

\begin{lemma} \label{lemma:Hom-finiteness}
If $R^G$ has isolated singularities, then the category $\sg R^G$ is $\Hom$-finite.
\end{lemma}

\begin{proof}
Since $R^G$ has isolated singularities, so does its localization $R^G_\mathfrak{p}$ at each prime ideal $\mathfrak{p}$. By Proposition~3.4 of \cite{Takahashi10}, its adic completion $\hat{R^G_\mathfrak{p}}$ also has isolated singularities. By Lemma~10.160.11 of \cite{Stacks26} and the implication from (d) to (b) in Theorem~1 of \cite{Auslander86}, the category $\sg \hat{R^G_\mathfrak{p}}$ is $\Hom$-finite. Since $G$ is a finite subgroup of $\SL_n$ and its order is not divisible by the characteristic of $k$, by Theorem~1 of \cite{Watanabe74}, the commutative ring $R^G$ is Gorenstein. Then by Proposition~3.5 of \cite{KalckKlapprothPavic24}, the idempotent completion of the category $\sg R^G$ is $\Hom$-finite and hence so is $\sg R^G$.
\end{proof}

\begin{proof}[Proof of Theorem~\ref{thm:SLn}]
a) The statement follows from the proof of part~a) of Theorem~\ref{thm:GLn} and Proposition~\ref{prop:SLn minimal model}.

b) If we have $n=1$, then the group $G$ is trivial and $e_0$ equals $1$. It follows that we have $\CM R^G\simeq \add k[x]$ and $\cc_{\Pi_n(Q, W),e_0}\liso \per k[x]$. So the category $\ch_{\Pi_n(Q, W),e_0}$ is equivalent to the full subcategory of $\per k[x]$ whose objects are the $X$ satisfying
\[
(\per k[x])(k[x], \Si^i X)=0=(\per k[x])(X, \Si^i k[x])
\]
for all positive integers $i$. Since the algebra $k[x]$ is hereditary, any object in $\per k[x]$ is isomorphic to the direct sum of stalk complexes. Then it is easy to see that the above full subcategory of $\per k[x]$ is $\add k[x]$. Therefore, the statements follow.

From now on, we assume that we have $n\geq 2$. Since the order of $G$ is not divisible by the characteristic of $k$, the $R^G$-module $R$ is Cohen--Macaulay. By Lemma~\ref{lemma:Hom-finiteness}, we deduce that the algebra $\End_{\ul{\mathrm{CM}}\, R^G}(R)$ is finite-dimensional, where $\ul{\mathrm{CM}}\, R^G$ denotes the ideal quotient of the category $\CM R^G$ by the ideal of the morphisms which factor through a projective $R^G$-module. Since $G$ is a finite subgroup of $\SL_n$ and its order is not divisible by the characteristic of $k$, by Theorem~5.15 of \cite{LeuschkeWiegand12}, we have an isomorphism
\[
G\# R\xlongrightarrow{_\sim} \End_{R^G}(R)
\]
of algebras. It induces the isomorphism
\[
(G\# R)/e_0 \xlongrightarrow{_\sim} \End_{\ul{\mathrm{CM}}\, R^G}(R)
\]
of algebras. We write $\Pi$ for $\Pi_n(Q, W)$ and $\Pi'$ for $\Pi_n(Q', W')$. By the definition of $(Q', W')$ and Proposition~\ref{prop:SLn minimal model}, we have the isomorphism
\[
H^0(\Pi')\xlongrightarrow{_\sim} e(G\# R)e/e_0
\]
of algebras. By the above isomorphisms, we conclude that $H^0(\Pi')$ is finite-dimensional. Since the commutative ring $R^G$ is noetherian and the $R^G$-module $L\ten R$ is finitely generated, \linebreak so is its submodule $(L\ten R)^G$. Thus, by Lemma~\ref{lemma:invariant subalgebra}, we deduce that the subcategory $\add e_0 \Pi \subseteq \add \Pi$ is functorially finite. By Proposition~\ref{prop:SLn minimal model}, the homology of $\Pi$ is concentrated in degree $0$. Therefore, by parts~(1) and (3) of Corollary~3.26 of \cite{Wu25}, the Higgs category $\ch_{\Pi, e_0}$ is a Karoubian Frobenius exact category with the full subcategory $\add e_0 \Pi$ of projective-injective objects and the endomorphism algebra of the object $\Pi$ is isomorphic to $H^0(\Pi)$. Then by Theorem~2.7 of \cite{KalckIyamaWemyssYang15}, we have an equivalence
\[
\CM R^G \iso \ch_{\Pi, e_0}
\]
of exact categories mapping $(L\ten R)^G$ to $\Pi$ and it induces a triangle equivalence
\[
\sg R^G \iso \cc_{\Pi'} \: .
\]
By part~(1) of Corollary~3.26 of \cite{Wu25}, \cf also part~1) of Theorem~2.2 of \cite{Guo11}, the triangulated category $\cc_{\Pi'}$ is $(n-1)$-Calabi--Yau and hence so is $\sg R^G$. The last statement follows from part~(1) of Corollary~3.26 of \cite{Wu25}, \cf also part~2) of Theorem~2.2 of \cite{Guo11}.
\end{proof}


\begin{thebibliography}{10}

\bibitem{Amiot09}
Claire Amiot, \emph{Cluster categories for algebras of global dimension 2 and
  quivers with potential}, Ann. Inst. Fourier (Grenoble) \textbf{59} (2009),
  no.~6, 2525--2590.

\bibitem{AmiotIyamaReiten15}
Claire Amiot, Osamu Iyama, and Idun Reiten, \emph{Stable categories of
  {C}ohen-{M}acaulay modules and cluster categories}, Amer. J. Math.
  \textbf{137} (2015), no.~3, 813--857.

\bibitem{Auslander86}
Maurice Auslander, \emph{Isolated singularities and existence of almost split
  sequences}, Representation theory, {II} ({O}ttawa, {O}nt., 1984), Lecture
  Notes in Math., vol. 1178, Springer, Berlin, 1986, pp.~194--242.

\bibitem{BocklandtSchedlerWemyss10}
Raf Bocklandt, Travis Schedler, and Michael Wemyss, \emph{Superpotentials and
  higher order derivations}, J. Pure Appl. Algebra \textbf{214} (2010), no.~9,
  1501--1522.

\bibitem{Booth21}
Matt Booth, \emph{Singularity categories via the derived quotient}, Adv. Math.
  \textbf{381} (2021), Paper No. 107631, 56.

\bibitem{Buchweitz21}
Ragnar-Olaf Buchweitz, \emph{Maximal {C}ohen-{M}acaulay modules and {T}ate
  cohomology}, Mathematical Surveys and Monographs, vol. 262, American
  Mathematical Society, Providence, RI, [2021] \copyright 2021, With appendices
  and an introduction by Luchezar L. Avramov, Benjamin Briggs, Srikanth B.
  Iyengar and Janina C. Letz.

\bibitem{ThanhofferVandenBergh16}
Louis de~Thanhoffer de~V\"olcsey and Michel Van~den Bergh, \emph{Explicit
  models for some stable categories of maximal {C}ohen-{M}acaulay modules},
  Math. Res. Lett. \textbf{23} (2016), no.~5, 1507--1526.

\bibitem{Ginzburg06}
Victor Ginzburg, \emph{{Calabi-Yau} algebras}, arXiv:math/0612139v3 [math.AG].

\bibitem{Guo11}
Lingyan Guo, \emph{Cluster tilting objects in generalized higher cluster
  categories}, J. Pure Appl. Algebra \textbf{215} (2011), no.~9, 2055--2071.

\bibitem{IyamaTakahashi13}
Osamu Iyama and Ryo Takahashi, \emph{Tilting and cluster tilting for quotient
  singularities}, Math. Ann. \textbf{356} (2013), no.~3, 1065--1105.

\bibitem{KalckIyamaWemyssYang15}
Martin Kalck, Osamu Iyama, Michael Wemyss, and Dong Yang, \emph{Frobenius
  categories, {G}orenstein algebras and rational surface singularities},
  Compos. Math. \textbf{151} (2015), no.~3, 502--534.

\bibitem{KalckKlapprothPavic24}
Martin Kalck, Carlo Klapproth, and Nebojsa Pavic, \emph{Obstructions to
  semiorthogonal decompositions for singular projective varieties {II}:
  Representation theory}, arXiv:2404.07816 [math.AG].

\bibitem{KalckYang18}
Martin Kalck and Dong Yang, \emph{Relative singularity categories {II}: {DG}
  models}, arXiv:1803.08192 [math.AG].

\bibitem{Keller24}
Bernhard Keller, \emph{On {H}iggs categories for cluster algebras arising in
  higher {T}eichm\"uller theory}, International {C}onference on
  {R}epresentations of {A}lgebras, Shanghai 2024, E{MS} {P}ress, to appear.

\bibitem{Keller11b}
\bysame, \emph{Deformed {C}alabi-{Y}au completions}, J. Reine Angew. Math.
  \textbf{654} (2011), 125--180, With an appendix by Michel Van den Bergh.

\bibitem{KellerMurfetVandenBergh11}
Bernhard Keller, Daniel Murfet, and Michel Van~den Bergh, \emph{On two examples
  by {I}yama and {Y}oshino}, Compos. Math. \textbf{147} (2011), no.~2,
  591--612.

\bibitem{KellerWu23}
Bernhard Keller and Yilin Wu, \emph{Relative cluster categories and {H}iggs
  categories with infinite-dimensional morphism spaces}, with an appendix by
  Chris Fraser and Bernhard Keller, arXiv:2307.12279 [math.RT].

\bibitem{Lam14}
Yan~Ting Lam, \emph{Calabi-{Y}au {C}ategories and {Q}uivers with
  {S}uperpotential}, University of {O}xford, 2014,
  https://people.maths.ox.ac.uk/joyce/theses/LamDPhil.pdf.

\bibitem{LeuschkeWiegand12}
Graham~J. Leuschke and Roger Wiegand, \emph{Cohen-{M}acaulay representations},
  Mathematical Surveys and Monographs, vol. 181, American Mathematical Society,
  Providence, RI, 2012.

\bibitem{McConnellRobson01}
J.~C. McConnell and J.~C. Robson, \emph{Noncommutative {N}oetherian rings},
  revised ed., Graduate Studies in Mathematics, vol.~30, American Mathematical
  Society, Providence, RI, 2001, With the cooperation of L. W. Small.

\bibitem{McKay83}
J.~McKay, \emph{Graphs, singularities and finite groups}, Uspekhi Mat. Nauk
  \textbf{38} (1983), no.~3(231), 159--162, Translated from the English by S.
  A. Syskin.

\bibitem{Orlov04}
D.~O. Orlov, \emph{Triangulated categories of singularities and {D}-branes in
  {L}andau-{G}inzburg models}, Tr. Mat. Inst. Steklova \textbf{246} (2004),
  240--262.

\bibitem{Orlov06}
\bysame, \emph{Triangulated categories of singularities, and equivalences
  between {L}andau-{G}inzburg models}, Mat. Sb. \textbf{197} (2006), no.~12,
  117--132.

\bibitem{Orlov11}
Dmitri Orlov, \emph{Formal completions and idempotent completions of
  triangulated categories of singularities}, Adv. Math. \textbf{226} (2011),
  no.~1, 206--217.

\bibitem{PavicShinder21}
Nebojsa Pavic and Evgeny Shinder, \emph{K-theory and the singularity category
  of quotient singularities}, Ann. K-Theory \textbf{6} (2021), no.~3, 381--424.

\bibitem{Plamondon11}
Pierre-Guy Plamondon, \emph{Cluster characters for cluster categories with
  infinite-dimensional morphism spaces}, Adv. Math. \textbf{227} (2011), no.~1,
  1--39.

\bibitem{Stacks26}
The~Stacks project authors, \emph{The {S}tacks project}, accessed 2026,
  https://stacks.math.columbia.edu/.

\bibitem{Takahashi10}
Ryo Takahashi, \emph{Classifying thick subcategories of the stable category of
  {C}ohen-{M}acaulay modules}, Adv. Math. \textbf{225} (2010), no.~4,
  2076--2116.

\bibitem{VandenBergh15}
Michel Van~den Bergh, \emph{Calabi-{Y}au algebras and superpotentials}, Selecta
  Math. (N.S.) \textbf{21} (2015), no.~2, 555--603.

\bibitem{Watanabe74}
Keiichi Watanabe, \emph{Certain invariant subrings are {G}orenstein. {I},
  {II}}, Osaka Math. J. \textbf{11} (1974), 1--8; ibid. 11 (1974), 379--388.

\bibitem{Wu25}
Yilin Wu, \emph{Silting reduction, relative {AGK}'s construction and {H}iggs
  construction}, arXiv:2510.00470 [math.RT].

\bibitem{Wu23a}
\bysame, \emph{Relative cluster categories and {H}iggs categories}, Adv. Math.
  \textbf{424} (2023), Paper No. 109040, 112.

\end{thebibliography}


\def\cprime{$'$} \def\cprime{$'$}
\providecommand{\bysame}{\leavevmode\hbox to3em{\hrulefill}\thinspace}
\providecommand{\MR}{\relax\ifhmode\unskip\space\fi MR }
\providecommand{\MRhref}[2]{%
  \href{http://www.ams.org/mathscinet-getitem?mr=#1}{#2}
}
\providecommand{\href}[2]{#2}

\end{document}